\documentclass{amsart}
\usepackage{amsmath}
  \usepackage{paralist}
  \usepackage{graphics} 
  \usepackage{epsfig} 
 \usepackage[colorlinks=true]{hyperref}
\hypersetup{urlcolor=blue, citecolor=red}

\newtheorem{theorem}{Theorem}[section]
\newtheorem{corollary}{Corollary}

\newtheorem{lemma}[theorem]{Lemma}
\newtheorem{proposition}{Proposition}

\theoremstyle{definition}
\newtheorem{definition}[theorem]{Definition}
\newtheorem{remark}{Remark}

\title[A variational formulation via bipotentials]
      {A variational formulation for constitutive laws described by bipotentials}

\author[Marius Buliga, G\'ery de Saxc\'e and Claude Vall\'ee]{}


 \email{Marius.Buliga@imar.ro}
 \email{gery.desaxce@univ-lille1.fr}
 \email{vallee@lms.univ-poitiers.fr}

\thanks{The first author is partially supported by the grant 
"Continuous modeling of advanced materials in virtual fabrication" 
COMOD PCCE - ID 100. All authors acknowledge the support from the European 
Associated Laboratory "Math Mode"  associating the 
Laboratoire de Math\'ematiques de l'Universit\'e Paris-Sud (UMR 8628) and the 
"Simion Stoilow" Institute of Mathematics of the Romanian Academy.}

\begin{document}
\maketitle

\centerline{\scshape Marius Buliga}
\medskip
{\footnotesize
 \centerline{"Simion Stoilow" Institute of Mathematics of the Romanian Academy,}
   \centerline{ PO BOX 1-764,014700 Bucharest, Romania}
} 

\medskip

\centerline{\scshape G\'ery de Saxc\'e}
\medskip
{\footnotesize
 \centerline{Laboratoire de M\'ecanique de Lille, UMR CNRS 8107,}
   \centerline{Universit\'e des Sciences et 
  Technologies de Lille,}
   \centerline{B\^atiment Boussinesq, Cit\'e Scientifique, 59655 Villeneuve d'Ascq cedex, 
 France}
}

\medskip

\centerline{\scshape Claude  Vall\'ee}
\medskip
{\footnotesize
 \centerline{Institut Pprime, UPR CNRS 3346,}
   \centerline{ bd M. et P. Curie, t\'el\'eport 2, BP 30179, 86962 Futuroscope-Chasseneuil cedex, 
France}
}

\bigskip

\begin{abstract}
Inspired by the  algorithm for solving the discretisation in time of the 
evolution problem for an implicit standard material,  
presented in  \cite{bergadesaxce}, we propose a variational formulation in terms
of bipotentials. 
\end{abstract}

\section{Introduction}

In the paper \cite{bergadesaxce} Berga and de Saxc\'e propose a bipotential 
for the constitutive law of a soil and further they proceed with the 
variational formulation of this model. 

We are interested in the precise formulation of the model, especially we want to understand from a mathematical viewpoint the recipe proposed in \cite{bergadesaxce} for using bipotentials in order to get a variational formulation of their model. We regard this as the first step towards the 
establishment of a general variational theory of bipotentials.

The following paragraph, extracted from \cite{bergadesaxce} page 414, is 
revealing for two reasons:  (a) the understanding of the motivation for 
introducing the bifunctional in order to adapt the Uzawa algorithm for implicit 
constitutive laws; (b) the imprecision concerning the understanding of the 
proposed new algorithm, related to the fact that, as we shall see, the simultaneous minimization of the bifunctional is not in fact how the algorithm 
works. 

"{\em One of the advantages of the new formulation is  to extend the classical 
Calculus of Variations to non associated constitutive laws. In the theoretical frame of the Implicit Standard Materials, a new functional, called bifunctional, 
is introduced, depending on both the displacement and stress field. The exact 
solution of the Boundary Value Problem corresponds to the simultaneous minimization of the bifunctional, firstly with respect to kinematically admissible displacement fields, when the stress field is equal with the exact one, and secondly with respect to statically admissible stress fields, when the displacement field is the exact one. The two minimization problems are the direct extension of the dual variational principles of displacements and stresses.}"

The notion of bipotential (definition \ref{def2}) has been introduced in 
\cite{saxfeng}, in order to formulate a large family of non associated 
constitutive laws in terms of convex analysis. The basic idea is explained 
further in few words. In Mechanics the associate constitutive laws are simply 
relations $y \in \partial \phi (x)$, with $\phi: X \rightarrow \mathbb{R} \cup 
\left\{ + \infty \right\}$ a convex and lower semicontinuous function. By
Fenchel inequality such a relation is equivalent with $\displaystyle 
\phi(x) + \phi^{*}(y) = \langle x, y \rangle$, where $\phi^{*}$ is the Fenchel
conjugate of $\phi$. It has been noticed that often in the mathematical 
study of problems related to associated constitutive laws enters not the 
function $\phi$, but the expression 
$$ b(x,y) = \phi(x) + \phi^{*}(y)$$
which we call "separable bipotential". The idea is then to use as 
a basic notion the one of bipotential $b: X \times Y \rightarrow \mathbb{R} \cup 
\left\{ + \infty \right\}$, which is convex and lsc in each argument and
satisfies a generalization of the Fenchel inequality. To non associated
constitutive laws thus corresponds bipotentials which are not separable.

There are many  such laws which can 
be studied with the help of bipotentials, as witnessed by the papers listed
further. In many of these papers bipotentials are used for numerical purposes
and several ad hoc algorithms have been suggested and exploited for
applications.  Here is a partial list of constitutive laws which have been 
described by bipotentials: non-associated Dr\"ucker-Prager \cite{sax boush KIELCE 93}  and 
Cam-Clay models \cite{sax BOSTON 95} in soil mechanics, 
cyclic Plasticity (\cite{sax CRAS 92},\cite{bodo sax EJM 01}) 
and Viscoplasticity \cite{hjiaj bodo CRAS 00} of metals with non linear 
kinematical hardening rule, Lemaitre's damage law \cite{bodo}, the coaxial 
laws (\cite{sax boussh 2},\cite{vall leri CONST 05}), the Coulomb's friction 
law \cite{saxfeng}, \cite{sax CRAS 92}, \cite{boush chaa IJMS 02}, \cite{feng hjiaj CM 06}, \cite{fort hjiaj CG 02}, 
\cite{hjiaj feng IJNME 04}, \cite{sax boush KIELCE 93}, \cite{sax feng IJMCM
98}, \cite{laborde}, \cite{bipo3}.
A complete survey can be found in \cite{sax boussh 2}.  

Later we started in \cite{bipo1} \cite{bipo2} \cite{bipo3} a mathematical 
study of bipotentials and their relation with convex analysis. This paper is 
another contribution along this subject, concerning mathematically sound
variational formulations and algorithms for numerically solving the quasistatic
evolution problem for constitutive laws of implicit standard materials. For another paper which contains a
variational formulation via bipotentials for the particular case of separated
bipotentials, see \cite{niculescu}.

\section{Notations and prerequisites from convex analysis}
\label{secnot}

 $X$ and $Y$ are topological, locally convex, real vector spaces of dual 
variables $x \in X$ and $y \in Y$, with the duality product 
$\langle \cdot , \cdot \rangle : X \times Y \rightarrow \mathbb{R}$. 
We shall suppose that $X, Y$ have topologies compatible with the duality 
product, that is: any  continuous linear functional on $X$ (resp. $Y$) 
has the form $x \mapsto \langle x,y\rangle$, for some $y \in Y$ (resp. 
$y \mapsto \langle x,y\rangle$, for some  $x \in X$). 
We use the notations: 
\begin{enumerate} 
\item[-] $\displaystyle \bar{\mathbb{R}} = \mathbb{R}\cup 
\left\{ +\infty \right\}$; 
\item[-] the domain of a function  $\displaystyle \phi: X \rightarrow
\bar{\mathbb{R}}$ is $dom \, \phi = \left\{ x \in X \mbox{ : } \phi(x) \in
\mathbb{R} \right\}$; 
\item[-] $\displaystyle \Gamma_{0}(X) = \left\{ \phi: X \rightarrow
\bar{\mathbb{R}} 
\mbox{ : } \phi \mbox{ is lsc and } dom \phi \not = \emptyset \right\}$; 
\item[-] for any convex and closed set $A \subset X$, its  indicator function,  
$\displaystyle \Psi_{A}$, is defined by 
$$\Psi_{A} (x) = \left\{ \begin{array}{ll}
0 & \mbox{ if } x \in A \\ 
+\infty & \mbox{ otherwise } 
\end{array} \right. $$
\item[-] the subgradient of a function $\displaystyle \phi: X \rightarrow
\bar{\mathbb{R}}$ at a point $x \in X$ is the (possibly empty) set: 
$$\partial \phi(x) = \left\{ u \in Y \mid \forall z \in X  \  \langle z-x, u \rangle \leq \phi(z) - \phi(x) \right\} \  .$$ 
\item[-] the inf-convolution of two functions $\displaystyle \phi, \psi \in
\Gamma_{0}(X)$ is the function $\phi \square \psi \in \Gamma_{0}(X)$ defined by:
for any $x \in X$ 
$$\phi \square \psi (x) \, = \, \inf \left\{ \phi(u) + \psi(v) \mbox{ : } u+v=x
\right\}$$
\end{enumerate}

\section{Bipotentials and syncs}

\begin{definition} A bipotential is a function $b: X \times Y \rightarrow
 \bar{\mathbb{R}}$, with the properties: 
\begin{enumerate}
\item[(a)] for any $x \in X$, if $dom \, b(x, \cdot) \not = \emptyset$ then 
$\displaystyle b(x, \cdot) \in \Gamma_{0}(X)$;  for any $y \in Y$, if $dom \,
b(\cdot, y) \not = \emptyset$ then 
$\displaystyle b(\cdot, y) \in \Gamma_{0}(Y)$; 
 \item[(b)] for any $x \in X , y\in Y$ we have $\displaystyle b(x,y) \geq \langle x, y \rangle$; 
\item[(c)]  for any $(x,y) \in X \times Y$ we have the equivalences: 
\begin{equation}
y \in \partial b(\cdot , y)(x) \ \Longleftrightarrow \ x \in \partial b(x, \cdot)(y)  \ \Longleftrightarrow \ b(x,y) = 
\langle x , y \rangle \ .
\label{equiva}
\end{equation}
\end{enumerate}
The graph of $b$ is 
\begin{equation}
M(b) \ = \ \left\{ (x,y) \in X \times Y \ \mid \ b(x,y) = \langle x, y \rangle \right\} \  .
\label{mb}
\end{equation}
\label{def2}
\end{definition}

Bipotentials are related to syncronised convex functions, defined further. 

\begin{definition}
A sync (syncronised convex function) is a function $c: X \times Y 
[0,+\infty]$ with the properties: 
\begin{enumerate}
\item[(a)] for any $x \in X$, if $dom \, c(x, \cdot) \not = \emptyset$ then 
$\displaystyle c(x, \cdot) \in \Gamma_{0}(X)$;  for any $y \in Y$, if $dom \,
c(\cdot, y) \not = \emptyset$ then 
$\displaystyle c(\cdot, y) \in \Gamma_{0}(Y)$; 
\item[(b)] for any $x \in X$, if  $dom \, c(x,\cdot) \not = \emptyset$ and 
the minimum  $\min\left\{ c(x,y) \mbox{ : } y \in
Y \right\}$ exists then this minimum equals $0$; for any $y \in X$, if  $dom \,
c(\cdot, y) \not = \emptyset$ and 
the minimum  $\min\left\{ c(x,y) \mbox{ : } x \in
X \right\}$ exists then this minimum equals $0$. 
\end{enumerate}
\label{defsync}
\end{definition}

\begin{proposition}
A function $b: X \times Y \rightarrow
 \bar{\mathbb{R}}$ is a bipotential if and only if the function 
$c: X \times Y \rightarrow
 \bar{\mathbb{R}}$, $c(x,y) = b(x,y) - \langle x, y \rangle$ is a sync. 
\label{psync}
\end{proposition}

\begin{remark}
The string of equivalences (\ref{equiva}) justifies the name 
"syncronised convex function", as it expresses the fact that critical 
points of functions $c(x, \cdot)$ are related with critical points of 
functions $c(\cdot, y)$. 
\end{remark}

With the notations from proposition \ref{psync}, we have  
$\displaystyle M(b) = c^{-1}(0)$. Also, for any $x \in X$ and $y \in Y$,
property (a) definition \ref{defsync} of syncs is equivalent with: 
$$epi(c) \cap \left\{x\right\} \times Y \times \mathbb{R} \mbox{ and } 
epi(c) \cap X \times \left\{y\right\}  \times \mathbb{R}$$ 
are closed convex sets, where $epi(c)$ is the epigraph of $c$: 
$$epi(c) = \left\{ (x,y,r) \in X \times Y \times \mathbb{R} \mbox{ : } 
c(x,y) \leq r \right\}$$

An interesting fact is that duality products do not enter in the definition of syncs. As an application,  let $(X,Y)$ be a pair of spaces, $\langle \cdot , \cdot \rangle$ and $\langle \cdot , \cdot \rangle'$ be two duality products, defined on $X \times Y$, and $c: X \times Y \rightarrow 
[0, +\infty]$ be a sync. We define  the applications: 
$$b \, , \, b' : X \times Y \rightarrow \mathbb{R} \cup \left\{ + \infty\right\} \quad 
b(x,y) = c(x,y) + \langle x,y \rangle \, , \, b'(x,y) = c(x,y) + \langle x,y \rangle'$$
Then $b$ is a bipotential with respect to $\langle \cdot , \cdot \rangle$ and $b'$ is a bipotential 
with respect to $\langle \cdot , \cdot \rangle'$. As a corollary, if we have a bipotential $b$ with respect to the duality product $\langle \cdot , \cdot \rangle$ and $\langle \cdot , \cdot \rangle'$ 
is another duality product, then the application $b'$ defined by 
$$b'(x,y) = b(x,y) - \langle x,y \rangle \ + \langle x,y \rangle'$$
is a bipotential with respect to the duality product $\langle \cdot , \cdot \rangle'$ and 
$M(b) = M(b')$ (they describe the same law). More generally, we have the following proposition concerning transformations of syncs.

\begin{proposition}
Let $(X,Y,\langle \cdot , \cdot \rangle)$, $(X',Y',\langle \cdot , \cdot \rangle')$ be two pairs 
of spaces with their respective duality products, $T: X \rightarrow X'$ and $L: Y \rightarrow Y'$
be two linear, bijective, continuous transformations, $\alpha > 0$  and $c': X' \times Y' \rightarrow [0, +\infty]$ be a sync. 
Then the function 
$$c: X \times Y \rightarrow [0, +\infty] \quad , \quad c(x,y) \, = \, \alpha \, c'(T x , L y)$$
is a sync and $\displaystyle c^{-1}(0) = c'^{-1}(0)$. 
\label{pchangebipo}
\end{proposition}

\begin{proof} The application $c'$ is a sync, therefore it satisfies conditions (a), (b) from definition \ref{defsync}. It is straightforward that $c$ is convex and lsc in each argument, therefore condition (a) definition \ref{defsync} is a consequence of the same condition for $c'$. Also, because 
$T$ and $L$ are bijective,  condition (b) for the application $c$ follows from the same condition 
for $c'$. \end{proof}

 The following is definition 3.1 \cite{bipo1}.

\begin{definition}
A non empty set $M \subset X \times Y$ is a BB-graph (bi-convex and 
 bi-closed) if for any $x \in X$ and $y \in Y$ the sections 
$$\displaystyle M (x) \ = \ \left\{ y \in Y \mid (x,y) \in M \right\} \ $$  
$$ M^{*}(y) \ = \ \left\{ x \in X \mid (x,y) \in M \right\} \ $$
 are convex and closed. 
\end{definition}

For any  BB-graph $M$ the indicator function $\displaystyle \Psi_{M}$ is 
obviously a sync. 
 To this sync corresponds the bipotential  
 $$ b_{\infty} (x,y) = \left\langle x,y \right\rangle  + \Psi_M (x,y) .$$ 
In particular, this shows that to a BB-graph we may associate more than one 
bipotential. Indeed, if $M$ is maximal cyclically monotone then it is 
the graph of a separable bipotential, but also the graph of the bipotential 
associated to the sync $\displaystyle \Psi_{M}$ (that is a bipotential of the 
form $\displaystyle b_{\infty}$). Therefore maximal cyclically monotone 
graphs admit at least two distinct bipotentials.

\section{Implicit standard materials described by  bipotentials}

In the mechanics of standard materials, the evolution problem is generally 
given by a set of equations, inequations,  boundary and initial 
conditions. They can be structured in three groups: kinematical equations, equilibrium equations and  
the constitutive law modeling the material behavior.

\subsection{Notations.} The configuration of the body is represented by
$\Omega$, an open, bounded set 
with piecewise smooth boundary $\partial \Omega$. 

We denote by $n$ the dimension of the configuration space ($n=1,2$ or $3$), 
thus $\displaystyle \Omega \subset \mathbb{R}^{n}$. 

The boundary decomposes in 
two disjoint parts: on $\displaystyle \partial_{0} \Omega$ displacements are 
imposed, while given surface forces act on the remaining part of the boundary 
denoted by $\displaystyle \partial_{1} \Omega$. 
The closure of $\Omega$ is denoted by $\bar{\Omega}$. 

The following quantities are considered. 
\begin{enumerate}
\item[-] $u$ is the displacement of the body with respect to the 
configuration  $\Omega$,
\item[-]  $\displaystyle \varepsilon = D(u) = \frac{1}{2} \left( \nabla u + \nabla u^{T} \right)$ is the associated strain. The trace of the strain is denoted by 
$\displaystyle e_{m} = \,\frac{1}{n} \, tr \, \varepsilon$ and the strain deviator is 
$$e = \varepsilon -  e_{m} I_{n} $$
\item[-] The strain $\varepsilon$ decomposes additively into elastic and plastic strains 
$$ \varepsilon = \varepsilon^{e} + \varepsilon^{p}$$ 
The traces of elastic strain $\displaystyle \varepsilon^{e}$ and plastic strain $\displaystyle \varepsilon^{p}$  are  denoted respectively by $\displaystyle e^{e}_{m}$, $\displaystyle e^{p}_{m}$,  and their deviatoric parts are 
$\displaystyle e^{e}$, $\displaystyle e^{p}$ respectively, 
\item[-] The stress field is denoted by $\sigma$, its trace is the hydrostatic 
pressure $\displaystyle s_{m} = \, tr \, \sigma$ and $s$ denotes the stress 
deviator. 
\item[-] $S$ is the elasticity tensor modulus.
\item[-] The density of volumic forces is $\displaystyle f_{v}$; on $\displaystyle \partial_{1} \Omega$ act the surface forces with density 
$\displaystyle f_{s}$. The class of stress fields $\sigma$ which satisfy the
equilibrium equations: 
$$\displaystyle div \, \sigma + f_{v} = 0 \mbox{  in } \Omega \quad , \quad 
 \sigma \cdot n = f_{s} \mbox{  on } \partial_{1}\Omega$$
  is denoted by $\displaystyle SA(f_{v}, f_{s})$. 
\item[-] The imposed boundary displacements on $\displaystyle \partial_{0} \Omega$ are denoted by $\displaystyle \bar{u}$. In fact, it is 
useful for further computations to consider the imposed boundary  displacement $\displaystyle \bar{u}$ to be defined over 
all $\bar{\Omega}$. The class of displacements $\displaystyle u$, such that 
$\displaystyle u - \bar{u} = 0$ on $\partial \Omega$ (possibly in the sense of
trace) is denoted by $\displaystyle CA(\bar{u})$ and called the class of
displacements which are kinematically admissible with respect to $\displaystyle
\bar{u}$.
\end{enumerate}

Let $\displaystyle Sym(n)$ be the space of $n \times n$ real symmetric matrices  and $\displaystyle Sym_{0}(n) \subset Sym(n)$ the subspace of real symmetric matrices with null trace. 
The decomposition of a real symmetric matrix into 
hydrostatic and deviatoric parts can be expressed by the linear bijective 
transformations: 
$$ T_{1}: Sym(n) \rightarrow \mathbb{R} \times Sym_{0}(n) \quad , \quad
T_{1}(\varepsilon) = \left( \frac{1}{n} \, tr \, \varepsilon , \varepsilon - \frac{1}{n} (tr \,
 \varepsilon ) I_{n} \right)$$
 $$ T_{2}: Sym(n) \rightarrow \mathbb{R} \times Sym_{0}(n) \quad , \quad
T_{2}(\sigma) = \left( tr \, \sigma , \sigma - \frac{1}{n} (tr \,
 \sigma ) I_{n} \right)$$
With the notations previously made, for any strain value $\varepsilon  \in Sym(n)$, or for any 
stress value $\sigma  \in Sym(n)$, the decompositions in hydrostatic and deviatoric parts are: 
$$T_{1}(\varepsilon) = (e_{m}, e) \quad , \quad T_{2}(\sigma) = (s_{m}, s)$$
(In order to keep track of physical dimensions, we should introduce two spaces 
$Sym(n)$, one for strains and the other for stresses, or introduce units of 
measure, but we feel that such notations are  only making the presentation 
unnecessary complicated.)
 
We shall consider the following duality products: 
$$\langle \cdot , \cdot \rangle : Sym(n) \times Sym(n) \rightarrow \mathbb{R} \quad , \quad \langle \varepsilon , \sigma \rangle = \, tr \, (\varepsilon  \sigma)$$
 $$\langle \cdot , \cdot \rangle' : \left( \mathbb{R} \times Sym_{0}(n)\right) \times \left( \mathbb{R} \times Sym_{0}(n)\right) \rightarrow \mathbb{R} \quad , \quad \langle (e_{m}, e) , (s_{m}, s) \rangle' = e_{m} s_{m} + \langle e , s \rangle $$
Remark that the first duality product is the one entering in the formulation 
of the dissipation (as an integral over the body configuration $\Omega$ of $\displaystyle 
\langle \dot{\varepsilon^{p}} , \sigma$. The second duality product will be 
used for the plastic bipotential, see later for the example of  the 
Berga \& de Saxc\'e bipotential for the non-associative Dr\"ucker-Prager 
law. The relation between these dualities is: 
$$\langle \varepsilon, \sigma \rangle = \langle T_{1} (\varepsilon) ,
T_{2}(\sigma) \rangle'$$
therefore (by passing to associated syncs and back) we can easily transform
bipotentials expressed in coordinates $(\varepsilon, \sigma)$ into bipotentials
expressed in coordinates $\displaystyle ((e_{m},e),(s_{m},s))$.

The kinematical equations are: 
\begin{equation}
\varepsilon =  \frac{1}{2} \left( \nabla u + \nabla u^{T} \right) \quad , \quad
u \in CA(\bar{u})
\end{equation}

The equilibrium equations are: 
\begin{equation}
\sigma \in SA(f_{v}, f_{s})
\end{equation}

The constitutive equations (besides the additive decomposition of the strain
into elastic and plastic parts) are expressed with two bipotentials: the elastic
and the plastic bipotential respectively.

The elastic bipotential is defined by the elasticity tensor modulus and it has
the form:
\begin{equation}
b_{e}(\varepsilon^{e}, \sigma) = \frac{1}{2} \langle \varepsilon^{e} , S \varepsilon^{e}\rangle + \frac{1}{2} \langle S^{-1} \sigma , \sigma \rangle
\label{bipoel}
\end{equation}
The elastic bipotential is defined over pairs of dual variables (elastic strain,
stress). It is a separable bipotential, expresses as the sum of the (density of)
the elastic energy and it's dual. Moreover, this bipotential is quadratic in
each variable. 

The plastic bipotential 
\begin{equation}
b_{p} = b_{p}(\dot{\varepsilon}^{p}, \sigma)
\label{bipop}
\end{equation}
is defined over another pair of dual variables, namely
(plastic strain rate, stress). In the case of standard materials, the plastic bipotential is separated
(expressed as the sum of the plastic potential and it's dual). For implicit
standard constitutive laws which can be expressed by a bipotential (like for
example the non-associative Dr\"ucker-Prager law), the bipotential is not
separated. 

The constitutive equations are: 
\begin{equation}
\varepsilon = \varepsilon^{e} + \varepsilon^{p}
\label{decomp}
\end{equation}

\begin{equation}
\varepsilon^{e} \in \partial b_{e} (\varepsilon^{e}, \cdot) (\sigma)
\label{clawe}
\end{equation}

\begin{equation}
\dot{\varepsilon}^{p} \in \partial b_{p} (\dot{\varepsilon}^{p}, \cdot) (\sigma)
\label{clawp}
\end{equation}

The constitutive equation (\ref{clawe}) is equivalent with $\displaystyle
\varepsilon^{e} = S^{-1} \sigma$, which is a linear equation. 
In order to enhance the resemblance between  (\ref{clawe}) and (\ref{clawp}), 
we could differentiate with respect to time in the constitutive equation for 
$\displaystyle \varepsilon^{e}$ and then express the result with the help of the
elastic bipotential: 

\begin{equation}
\dot{\varepsilon}^{e} \in \partial b_{e} (\dot{\varepsilon}^{e}, \cdot) 
(\dot{\sigma})
\label{clawee}
\end{equation}

\section{Non-associated Dr\"ucker-Prager elasto-plasticity}

An important example of an implicit standard material is provided by 
the non-associated Dr\"ucker-Prager constitutive law. Here we follow the
presentation from \cite{bergadesaxce}. 

\subsection{Plastically admissible stresses.} The model is characterized by a 
Dr\"ucker-Prager plastic yielding surface. The set of plastically admissible stresses is the following cone: 
$$K_{stress} = \left\{ \sigma = \frac{1}{3} s_{m} I \, + \, s \mbox{ such that } 
\frac{1}{k_{d}} \| s \| + s_{m} \, tg \, \phi \leq c \right\}$$
Here $c$  is the cohesion, $\phi$ is the friction angle and $\displaystyle
k_{d}$ is a constant whose significance is explained in \cite{bergadesaxce}
section 3, relations (3.1), (3.2). 

We denote by $\displaystyle K_{stress}' = T(K_{stress})$ the same cone in
coordinates $(s_{m}, s)$ of the stresses. 

\subsection{Plastically admissible strain rates.} Let $\theta \in [0, \phi ]$ be the dilatancy angle 
(if $\theta = \phi$ then we are in the case of associated Dr\"ucker-Prager elastoplasticity). The set of admissible plastic strain rates is the cone: 
$$K_{strain} = \left\{ \dot{\varepsilon}^{p} = \frac{1}{3} \dot{e}^{p}_{m} I \, + \,  \dot{e}^{p} \mbox{ such that } 
k_{d} \, tg \, \theta \,  \| \dot{e}^{p} \|  \leq \dot{e}^{p}_{m} \right\}$$
We denote by $\displaystyle K_{strain}' = T(K_{strain})$ the same cone in the representation 
$(e_{m}, e)$ of the strains.  
\subsection{The flow rule.} The constitutive equation for the evolution of the plastic strain  has the following expression:
\begin{equation}
\left( \left( \dot{e}^{p}_{m} + k_{d} ( tg \, \phi \, - \, tg \, \theta) \|\dot{e}^{p}\| \right) ,  
\dot{e}^{p}\right) \, \in \, \partial    \Psi_{K_{stress}'} (s_{m},s)
\label{ef1}
\end{equation}

Theorems 4.1, 4.2 from \cite{bergadesaxce}  are collected into the following. 

\begin{theorem}
Let $\displaystyle b_{p}': \left( \mathbb{R} \times Sym_{0}(n)\right) \times \left( \mathbb{R} \times Sym_{0}(n)\right) \rightarrow \mathbb{R}\cup \left\{ + \infty \right\}$ be the function: 
\begin{equation}
 b_{p}'((e_{m}, e), (s_{m}, s)) =  \left\{ \begin{array}{ll} 
C_{1} e_{m} + C_{2} ( s_{m} - \frac{c}{tg \, \phi} ) 
\| e \| & \mbox{ if } (s_{m}, s) \in K_{stress}' \mbox{ and} \\
 &  (e_{m}, e) \in K_{strain}' \\
+ \infty & \mbox{ otherwise}
\end{array} \right.
\label{bepe}
\end{equation}
where $\displaystyle \| e\|$ is the norm defined by $\displaystyle \| e \|^{2} = \langle e , e \rangle$ and   the constants  $$ \displaystyle C_{1} = \frac{c}{tg \,  \phi} \quad , \quad C_{2} = k_{d} ( tg \, \theta - tg \, \phi )$$ are coming from the flow rule (\ref{ef1}).  Then: 
\begin{enumerate}
\item[(a)] $\displaystyle b_{p}'$ is a bipotential with respect to the duality 
product $\displaystyle \langle \cdot , \cdot \rangle'$, 
\item[(b)] the non-associated Dr\"ucker-Prager  constitutive equation for the 
evolution of the plastic strain (\ref{ef1}) can be expressed with the help of 
the bipotential $\displaystyle b_{p}'$ as 
$$ b_{p}' ( (\dot{e}^{p}_{m}, \dot{e}^{p}) , ( s_{m} , s)) = \langle (\dot{e}^{p}_{m}, \dot{e}^{p}) , (s_{m}, s) \rangle'$$
\end{enumerate}
\end{theorem}

As an application of proposition \ref{pchangebipo}, we obtain the following
characterization of the Dr\"ucker-Prager  constitutive law. 

\begin{corollary}
In the coordinates $\displaystyle (\varepsilon, \sigma)\in Sym(n) \times Sym(n)$, with the duality 
product $\displaystyle \langle \varepsilon , \sigma \rangle = tr \, (\varepsilon  \sigma)$, 
the non-associated Dr\"ucker-Prager  constitutive law (\ref{ef1}) can be expressed with the help 
of the bipotential 
$\displaystyle b_{p}: Sym(n) \times Sym(n) \rightarrow \mathbb{R}\cup \left\{ + \infty \right\}$ 
defined by: 
\begin{equation}
b_{p}(\varepsilon, \sigma) = \Psi_{K_{stress}} (\sigma) + \Psi_{K_{strain}} (\varepsilon) + 
C_{1} \, tr \, \varepsilon \,  + C_{2} ( tr \, \sigma \, - \frac{c}{tg \, \phi} ) 
\| \varepsilon - \frac{1}{n}(tr \, \varepsilon) I  \| \, - 
\label{bepetrue}
\end{equation}
$$ - \, \left( 1 - \frac{1}{n}\right)(tr \, \varepsilon ) (tr \, \sigma)$$
\end{corollary}

\begin{remark}
The term 
containing $\displaystyle C_{2}$ represents a coupling between the hydrostatic 
part of the stress and the deviatoric part of the strain (rate). 
If $\displaystyle C_{2} = 0$ then we get the associated Dr\"ucker-Prager  
constitutive law.  In this case the last term from the right hand side of the 
expression (\ref{bepetrue}) can be eliminated by modifying the cones 
$\displaystyle K_{stress}$ and $\displaystyle K_{strain}$. But if 
$\displaystyle C_{2} \not = 0$ such a modification cannot be made because of 
the coupling between deviatoric and hydrostatic parts, 
so  this last term in the expression of $\displaystyle b_{p}$ can not 
disappear by a modification of  the cones  $\displaystyle 
K_{stress}$ and $\displaystyle K_{strain}$. 
\end{remark}

\section{Time discretisation of the evolution problem}

Given as the initial data the displacement $\displaystyle u_{0}$ and the initial
plastic strain $\displaystyle \varepsilon^{p}_{0}$, the boundary data
$\displaystyle \bar{u} = \bar{u}(t), f_{s} = f_{s}(t)$, and the volume forces 
$\displaystyle f_{v} = f_{v}(t)$, for $t \in [0,T]$,   a solution of the 
evolution problem is a collection $\displaystyle (u, \varepsilon^{p},
\varepsilon^{e}, \sigma)$  of fields dependent on $t$, which satisfy the
kinematical, equilibrium, constitutive equations, as well as the initial and
boundary conditions.  

We want to give a variational formulation of the time discretisation of the
evolution problem. For this we consider a discretisation 
$$\left\{ t_{0} = 0 , t_{1}, ... , t_{N} = T \right\}$$ 
of the time interval $[0,T]$. For each $\displaystyle k = 0, ... ,N$ we denote
by $\displaystyle (u_{k}, \varepsilon^{p}_{k},\varepsilon^{e}_{k}, \sigma_{k})$ 
the unknowns at the moment $\displaystyle t_{k}$. We shall use also the
notation: for any $\displaystyle k = 0, ... ,N$, let 
$\displaystyle \Delta t_{k} = t_{k+1} - t_{k}$, $\displaystyle \Delta u_{k} =
u_{k+1} - u_{k}$, and so on, for all fields, known or unknown. 

Further on, we shall replace the time derivatives from the evolution equation by
finite differences with respect to the considered time discretisation. The
problem which we want to solve is the following one. 

\subsection{ Problem (Pdisc).} Given  $\displaystyle
(u_{k}, \varepsilon^{p}_{k},\varepsilon^{e}_{k}, \sigma_{k})$, find 
$\displaystyle (\Delta u, \Delta \varepsilon^{p},\Delta \varepsilon^{e}, \Delta
\sigma)$, solution of the following problem: 

\begin{equation}
\Delta \varepsilon^{e} + \Delta \varepsilon^{p} = D \left( \Delta u \right)
\label{pdisc1}
\end{equation}

\begin{equation}
\Delta \varepsilon^{e} = S \Delta \sigma
\label{pdisc2}
\end{equation}

\begin{equation}
\frac{1}{\Delta t_{k}} \Delta \varepsilon^{p} \in \partial 
b_{p} \left(\frac{1}{\Delta t_{k}} \Delta \varepsilon^{p}, \cdot\right) (\sigma_{k} +
\Delta \sigma)
\label{pdisc3}
\end{equation}

\begin{equation}
\Delta \sigma \in SA(\Delta f_{v,k}, \Delta f_{s, k})
\label{pdisc4}
\end{equation}

\begin{equation} 
\Delta u \in CA(\Delta \bar{u}_{k})
\label{pdisc5}
\end{equation}

The unknowns $\displaystyle
(u_{k+1}, \varepsilon^{p}_{k+1},\varepsilon^{e}_{k+1}, \sigma_{k+1})$ are obtained as 
$$ u_{k+1} = u_{k} + \Delta u \quad , \quad \varepsilon^{p}_{k+1} =
\varepsilon^{p}_{k} + \Delta \varepsilon^{p} \quad , ... $$

Our first concern is to express (Pdisc) with the help of bipotentials. 

\begin{lemma}
For any $k = 0, ..., N-1$, the function 
\begin{equation}
 b_{p, k} (\Delta \varepsilon^{p}, \Delta \sigma) = \, 
 \Delta t_{k} b_{p} \left(\frac{1}{\Delta t_{k}} \Delta \varepsilon^{p}, \sigma_{k} +
\Delta \sigma \right) - \langle \Delta \varepsilon^{p}, \sigma_{k} \rangle
\label{peka}
\end{equation} 
is a bipotential and the equation (\ref{pdisc3}) is equivalent with 
\begin{equation}
\Delta \varepsilon^{p} \in \partial b_{p, k} (\Delta \varepsilon^{p}, \cdot 
) (\Delta \sigma)
\label{pdisc3p}
\end{equation}
\label{lem1}
\end{lemma}

\begin{proof}
Let us show that 
$$c_{p,k}(\Delta \varepsilon^{p}, \Delta \sigma) = \, b_{p, k} (\Delta
\varepsilon^{p}, \Delta \sigma) - \langle \Delta \varepsilon^{p}, \Delta \sigma
\rangle$$
is a sync. For this we introduce the sync associated to the bipotential
$\displaystyle b_{p}$, namely 
$$c_{p}(\Delta \varepsilon^{p}, \Delta \sigma) = \, b_{p} (\Delta
\varepsilon^{p}, \Delta \sigma) - \langle \Delta \varepsilon^{p}, \Delta \sigma
\rangle$$
Remark that 
$$c_{p,k}(\Delta \varepsilon^{p}, \Delta \sigma) = \, \Delta t_{k} \,
c_{p}\left(\frac{1}{\Delta t_{k}} \Delta \varepsilon^{p}, \sigma_{k} + \Delta
\sigma\right)$$
We apply proposition \ref{pchangebipo} and get the result. By consequence, the
function $\displaystyle b_{p,k}$ defined by (\ref{peka}) is a bipotential. 
From here, the second part of the proposition is a straightforward computation
which is left for the interested reader. \end{proof}

\subsection{Simplification of the boundary conditions and volume forces.} It is
not a restriction of generality to suppose that the boundary conditions and
volume forces are trivial, that is to suppose that equations 
(\ref{pdisc4}), (\ref{pdisc5}) have the following form: 

\begin{equation}
\Delta \sigma \in SA(0, 0)
\label{pdisc4p}
\end{equation}

\begin{equation} 
\Delta u \in CA(0)
\label{pdisc5p}
\end{equation}

 Indeed, let us
choose a field $\displaystyle \Delta \bar{\sigma} \in SA(\Delta f_{v,k}, \Delta
f_{s, k})$. If we define  the new unknowns: 
$$\Delta u' = \Delta u \, - \, \Delta \bar{u} \quad , \quad \Delta \sigma' \, =
\, \Delta \sigma - \Delta \bar{\sigma}$$
then we could use  using again proposition \ref{pchangebipo} in order to prove
that the constitutive equations, in the new unknowns, can be expressed by
bipotentials. 

In order not to use a too heavy notation, further on we shall assume
(\ref{pdisc4p}), (\ref{pdisc5p}) and we shall neglect the change of unknowns 
(thus maintaining the notations $\displaystyle \Delta u$, $\Delta \sigma$).

\subsection{Elimination of several unknowns.} We can simplify the problem (Pdisc)
by a standard argument involving the elimination of the unknowns $\displaystyle
\Delta \varepsilon^{e}, \Delta \varepsilon^{p}$, by using an inf-convolution. 

Indeed, let us denote $\displaystyle \Delta \varepsilon = \Delta \varepsilon^{e}
+ \Delta \varepsilon^{p}$. By equation (\ref{pdisc1}), $\displaystyle \Delta
\varepsilon$ can be deduced from $\displaystyle \Delta u$. 

For any $\displaystyle \Delta \sigma$, the functions $\displaystyle b_{e}(\cdot,
\Delta \sigma)$ and $\displaystyle b_{p,k}(\cdot, \Delta \sigma)$ are not
everywhere infinite, are convex and lower semicontinuous, therefore we can
define the inf-convolution of them: 
\begin{equation}
\Delta b_{k} (\Delta \varepsilon, \Delta \sigma) \, = \, \left( b_{e}(\cdot,
\Delta \sigma) \, \square \, b_{p,k}(\cdot, \Delta \sigma) \right) (\Delta
\varepsilon)
\end{equation}

\begin{lemma}
\begin{equation}
\Delta \sigma \in \partial \Delta b_{k} (\cdot, \Delta \sigma) (\Delta \varepsilon)
\label{pdiscinf}
\end{equation}
is equivalent with: there are $\displaystyle \Delta \varepsilon^{e}, \Delta
\varepsilon^{p}$, such that $\displaystyle \Delta \varepsilon = \Delta \varepsilon^{e}
+ \Delta \varepsilon^{p}$, which satisfy, together with $\displaystyle \Delta
\sigma$, the equations (\ref{pdisc2}), \ref{pdisc3}).
\label{lem2}
\end{lemma}

\begin{proof}
Indeed, by a well known property of inf-convolutions, equation (\ref{pdiscinf})
is equivalent with: there are $\displaystyle \Delta \varepsilon^{e}, \Delta
\varepsilon^{p}$, such that $\displaystyle \Delta \varepsilon = \Delta \varepsilon^{e}
+ \Delta \varepsilon^{p}$, which satisfy 
\begin{equation}
\Delta \sigma \in \partial  b_{e} (\cdot, \Delta \sigma) (\Delta \varepsilon^{e})
\label{pdiscinf1}
\end{equation}
\begin{equation}
\Delta \sigma \in \partial  b_{p,k} (\cdot, \Delta \sigma) (\Delta \varepsilon^{p})
\label{pdiscinf2}
\end{equation}
But both $\displaystyle b_{e}$ and $\displaystyle b_{p,k}$ are bipotentials,
therefore (\ref{pdiscinf1}) is equivalent with (\ref{pdisc2}) and
(\ref{pdiscinf2}) is equivalent with (\ref{pdisc3p}), which is equivalent with 
(\ref{pdisc3}) by lemma \ref{lem1}. \end{proof}

\begin{remark}
Because of the particular form of $\displaystyle b_{e}$ (quadratic function),
the inf-convolution $\displaystyle \Delta b_{k} (\cdot, \Delta \sigma)$ is 
differentiable, with Lipschitz gradient, as a kind of Moreau-Yosida regularization. Therefore the 
inclusion  (\ref{pdiscinf}) is equivalent with a standard equality, because 
the set from the right hand side contains only one element. This is an
 well known advantage of  this elimination of unknowns in associated plasticity.
 \end{remark}
 
Let us list the properties of the function $\displaystyle \Delta b_{k}$: 
\begin{enumerate}
\item[-] it is lower semicontinuous (even differentiable, with Lipschitz gradient
 in the first argument)
\item[-] $\displaystyle \Delta b_{k}$ is defined via an inf-convolution of a
bipotential of type (\ref{bepetrue}) with the
elastic bipotential $\displaystyle b_{e}$, therefore it satisfies the same
growth inequality as $\displaystyle b_{e}$ namely there is a constant $C > 0$  such that 
for any $\Delta \varepsilon \in Sym(n)$ and $\Delta \sigma \in Sym(n)$, if $\displaystyle 
 \Delta b_{k} (\Delta \varepsilon , \Delta \sigma) < + \infty$ then 
$$ \Delta b_{k} (\Delta \varepsilon , \Delta \sigma) \, \leq \, C \, \left( \| \Delta
\varepsilon \|^{2} + \| \Delta \sigma \|^{2} \right)$$
where $\displaystyle \| \cdot \|$ is an arbitrary euclidean norm on the
space $Sym(n)$, 
\item[-] it satisfies a weak form of the Fenchel inequality, 
$$\Delta b_{k} (\Delta \varepsilon , \Delta \sigma) \, \geq \, \langle \Delta
\varepsilon , \Delta \sigma \rangle$$ 
\item[-] it is convex in the first argument, but not in the second, therefore it
is not a bipotential, as it is stated in Theorem 6.1 \cite{bergadesaxce}. Remark
however that  the proof of Lemma \ref{lem2} uses the fact that the function
 $\displaystyle  b_{p,k}$ is a bipotential.
\end{enumerate}

We collect the partial results obtained so far into the following theorem, 
which provides a simplified  form of the problem (Pdisc). 

\begin{theorem}
The problem (Pdisc) is equivalent with the following one: find $\displaystyle 
(\Delta u, \Delta \sigma) \in  CA(0) \times SA(0,0)$ which satisfy
(\ref{pdiscinf}).
\end{theorem}

\section{Variational formulation of the problem (Pdisc)}

We give further a variational formulation \`a la Nayroles \cite{nayroles} of the
following general problem, which contains (Pdisc) as a particular case. 

We consider a first pair of spaces in duality: 
\begin{enumerate}
\item[-] $X = L^{2}(\Omega, Sym(n))$ is the space of the deformation fields 
$\varepsilon$, 
\item[-]  $Y = L^{2}(\Omega, Sym(n))$ is the space of stress fields $\sigma$.
\end{enumerate}
Instead of equalities $\displaystyle X, Y = L^{2}(\Omega, Sym(n))$, we may
consider that $X$ and $Y$ are topological, locally convex, real vector spaces 
of dual variables $\varepsilon \in X$ and $\sigma \in Y$, with the duality 
product $\displaystyle \langle \cdot , \cdot \rangle_{1} : X \times Y \rightarrow \mathbb{R}$, 
endowed with two injective continuous linear transformations 
$\displaystyle A: X  \rightarrow  L^{2}(\Omega, Sym(n))$ and 
$\displaystyle B: X  \rightarrow  L^{2}(\Omega, Sym(n))$ such that 
$$\displaystyle \langle \varepsilon, \sigma \rangle_{1} = \int_{\Omega} 
 \langle A(\varepsilon)(x),  B(\sigma)(x)\rangle \mbox{ d} x \, = \langle
 A(\varepsilon), B(\sigma) \rangle$$
In the integral we see the duality product (scalar product) on the space
$Sym(n)$ of $n \times n$ symmetric real matrices. In the right hand side we see
the duality product (scalar product) of $\displaystyle L^{2}$ with itself.

The space $X$, $Y$ may be finite dimensional (for example associated with a
discretisation in space by finite elements)  or infinite dimensional. In the
following we shall omit to mention the injections $A, B$ or any other similar
transformations which may appear. As an exception,  in the following 
theorem \ref{mainthm}, part (I), we need the spaces $X,Y$ to be "large enough"
in order to be able to prove that a solution of the variational formulation is
also a solution (almost everywhere) of the original problem. 
 
 $U$ is the space of $CA(0)$ displacement fields $\displaystyle u \in W^{1,2}(\Omega, \mathbb{R}^{n})$,  with 
 and $\displaystyle u = 0$ on
$\displaystyle \partial_{0} \Omega$ in the sense of trace. The linear
transformation $D: U \rightarrow X$, 
$\displaystyle \varepsilon = D(u) = \frac{1}{2} \left( \nabla u + \nabla u^{T}
\right)$ is continuous and $V = D(U) \subset X$ is the image.  

The space $\displaystyle Y_{0} \subset Y$ of statically admissible $SA(0,0)$
 stresses appear as the space of $\sigma \in Y$ with the property that for any $u \in U$ we have 
$$\langle D(u), \sigma \rangle_{1} = 0$$

We consider a function $\displaystyle b: Sym(n) \times K \rightarrow
\mathbb{R}$ with the following properties: 
\begin{enumerate}
\item[(a)] $K \subset Sym(n)$ is a closed convex set of the form $\displaystyle 
K = a + K_{0}$, with
$a \in Sym(n)$ and $\displaystyle K_{0} \subset Sym(n)$ a closed convex cone, 
such that $0 \in K$ (this is the set of plastically admissible stresses, as in the definition of the
Dr\"{u}cker-Prager plasticity). Let 
$\displaystyle \pi_{K}: Sym(n) \rightarrow K$ be the projection on this cone;  
\item[(b)] $b$  is lower semicontinuous in both arguments, differentiable with
Lipschitz gradient  and convex
in the first argument; moreover we suppose that the Lipschitz constant of the 
gradient of $b$ in the first
argument is continuous with respect to the second variable;  
\item[(c)] $b$ satisfies, for any $\varepsilon, \sigma \in Sym(n)$, the inequality: 
$b(\varepsilon, \sigma) \geq \langle \varepsilon, \sigma \rangle$,
\item[(d)] there is a constant $C > 0$  such that 
for any $\varepsilon \in Sym(n)$ and $\sigma \in K$ we have 
 \begin{equation}
  b (\varepsilon , \sigma) \, \leq \, C \, \left( \| \varepsilon \|^{2} + 
  \| \sigma \|^{2} \right)
  \label{growth}
  \end{equation}
\end{enumerate}
Associated to the function $b$ is the "bifunctional" of Berga and de Saxc\'e: 
$$B(\varepsilon, \sigma) \, = \, \int_{\Omega} b(\varepsilon(x), \sigma(x)) \mbox{
d}x$$

Our main theorem is the following: 

\begin{theorem} 
Suppose that the function $b$ takes only finite values, that is 
for any $(\varepsilon, \sigma) \in Sym(n) \times Sym(n)$ we have $b(\varepsilon,
\sigma) < + \infty$. 

(I) Let $u \in U$ and $\displaystyle \sigma \in Y_{0}$.  
The pair $(u, \sigma)$ satisfies 
almost everywhere in $\Omega$ 
$$\sigma \in \partial b(\cdot , \sigma) (D(u))$$
if and only if  for any $\varepsilon \in X$ we have 
\begin{equation}
B(D(u), \sigma) \, \leq \, B(\varepsilon, \sigma) - \langle \varepsilon, \sigma
\rangle_{1}
\label{maineq1}
\end{equation}

(II) For any $\displaystyle u^{0} \in  U$,  $\sigma^{0} \in  Y_{0}$  there is a 
sequence $\displaystyle (u^{k}, \sigma^{k})_{k}$ in  $\displaystyle U \times
Y_{0}$, such that for any $k \in \mathbb{N}$: 
\begin{enumerate}
  \item[a.](global condition) for all $v \in U$ the displacement
    $\displaystyle u^{k+1} \in U $ satisfies
 $$ B(D(u^{k+1}), \sigma^{k}) \leq B(D(v), \sigma^{k})$$ 
  \item[b.](local condition) the stress $\displaystyle \sigma^{k+1}$ satisfies
  almost everywhere in $\Omega$ the relation $$\displaystyle \sigma^{k+1} 
  \in  \partial 
  b(\cdot , \sigma^{k})(D(u^{k+1}))$$  
  \end{enumerate}
  
(III) If a sequence $\displaystyle (u^{k}, \sigma^{k})_{k}$ from (II) has a
subsequence (denoted by same symbols) such that $\displaystyle u^{k}$ converges
weakly in $\displaystyle W^{1,2}$ to $u$ and $\sigma^{k}$ converges weakly in
$L^{2}$ to $\sigma$, then $(u,\sigma)$ is a solution of the problem
(\ref{maineq1}).
\label{mainthm}
\end{theorem}

\begin{proof}
(I) We follow the convention: we identify an element of $\displaystyle 
g \in L^{2}(\Omega, Sym(n))$ (which is an equivalence class of functions) 
with its representant, defined almost everywhere in $\Omega$ by Lebesgue
theorem.   

Let $u \in U$ and $\displaystyle \sigma \in Y_{0}$, such that we have $\sigma
\in \partial b(\cdot , \sigma) (D(u))$ almost everywhere in $\Omega$. 
Let us take $\varepsilon \in X$. Then, by integration of the constitutive relation (and by the definition
of $\displaystyle Y_{0}$), we have 
$$\int_{\Omega} b(\varepsilon(x), \sigma(x)) \mbox{ d}x \, - \int_{\Omega}
\langle \varepsilon(x) , \sigma(x) \rangle \mbox{ d}x \, \geq 
\int_{\Omega} b(D(u)(x), \sigma(x)) \mbox{ d} x$$
which is exactly the relation (\ref{maineq1}). 

Conversely, let us start from the last integral inequality, supposed to be true
for any $\varepsilon \in X$. Further we suppose also that $\displaystyle X =
L^{2}(\Omega, Sym(n))$. Let us pick an arbitrary 
$\displaystyle x_{0}$ in the
intersection of the Lebesgue sets of $D(u)$ and $\sigma$.  
For any open ball $\displaystyle B(x_{0}, r) \subset \Omega$ 
centered in $\displaystyle x_{0}$
 we define $\varepsilon_{r} \in X$ such that  $\varepsilon_{r} = D(u)$ almost everywhere outside $B$. We
 obviously get that 
 $$\frac{1}{\mid B(x_{0},r)\mid}\left(\int_{B(x_{0}, r)} b(\varepsilon_{r}(x), 
 \sigma(x)) \mbox{ d}x \, - \, \int_{B(x_{0}, r)}
\langle \varepsilon_{r}(x) - D(u)(x) , \sigma(x) \rangle \mbox{ d}x  \right) \, \geq $$ 
\begin{equation}
\, \geq \, 
\frac{1}{\mid B(x_{0},r)\mid} \int_{B(x_{0}, r)} b(D(u)(x), \sigma(x)) \mbox{ d} x
\label{needi}
\end{equation}
For any $\displaystyle \bar{\varepsilon} \in Sym(n)$ we can choose  for any $
r>0$ (but sufficiently small) an $\displaystyle \varepsilon_{r}$ such that 
$$ \lim_{r \rightarrow 0} \frac{1}{\mid B(x_{0},r)\mid} \int_{B(x_{0},r)}
\varepsilon_{r}(x)  \mbox{ d}x \, = \, \bar{\varepsilon}$$
and such that we can  pass to the limit with $r$ to $0$, to obtain: 
$$b(\bar{\varepsilon}, \sigma(x_{0})) - \langle \bar{\varepsilon} - D(u)(x_{0}),
\sigma(x_{0}) \rangle \, \geq \, b(D(u)(x_{0}), \sigma(x_{0}))$$
This is equivalent with the satisfaction of the constitutive relation almost
everywhere. 

(II) Suppose that for $k \in \mathbb{N}$ we defined the element $\displaystyle
(u^{k}, \sigma^{k})$ of the sequence. We want to prove the existence of 
$\displaystyle (u^{k+1}, \sigma^{k+1})$ which satisfy the global condition (a),
 the local condition (b) and $\displaystyle \sigma^{k+1} \in Y_{0}$. 
 
 By the convexity, growth and continuity conditions on $b$, we easily obtain the
 existence of a minimizer of the functional  $\displaystyle 
 u \in U \, \mapsto \, B(D(u), \sigma^{k})$. This proves the existence of 
 $\displaystyle u^{k+1}$ which satisfy the global condition (a). The local
 condition (b) is in fact the definition of $\displaystyle \sigma^{k+1}$.
 Because of the differentiability and continuity conditions on $b$, it easily
 follows from $\displaystyle \sigma^{k} \in Y$ that $\sigma^{k+1} \in Y$. We
 have to prove that $\displaystyle \sigma^{k+1} \in Y_{0}$. For this, we choose
 an arbitrary $v \in U$ and we integrate the local condition (b). We obtain that
 $$B(D(v), \sigma^{k}) - B(D(u^{k+1}), \sigma^{k}) \, \geq \, \langle D(v) -
 D(u^{k+1}), \sigma^{k+1} \rangle$$
 The left hand side of this inequality is non negative (by the global condition)
 and it can be made arbitrarily small, for example by choosing 
 $\displaystyle v = u^{k+1} + \lambda w$, for a given, but arbitrary $w \in U$
 and $\lambda > 0$ smaller and smaller. As a conclusion we obtain
 that for any $w \in U$ we have 
 $$\langle D(w) , \sigma^{k+1} \rangle \leq 0$$ 
 which implies that $\displaystyle \sigma^{k+1} \in Y_{0}$. 
 
 (III) Suppose that $\displaystyle (u^{k}, \sigma^{k})$ converges, in the given
 sense, to $(u, \sigma)$. The sequence of functionals $v \mapsto B(D(v),
 \sigma^{k})$ converges in the variational sense to the functional 
 $v \mapsto B(D(v),  \sigma)$, so, up to the choice of a subsequence, the
 minimizers of these respective functionals (namely the $\displaystyle u^{k+1}$)
 converge to a minimizer of the latter functional. Therefore $(u,\sigma)$
 satisfy the condition 
 $$B(D(v), \sigma) \geq B(D(u), \sigma)$$
 for any $v \in U$. 
 
 The limit $\sigma$ is in $\displaystyle Y_{0}$ by construction. We can also
 pass to the limit in the integral form of the local condition, which is: for
 any $\displaystyle \varepsilon \in X$ 
 $$B(\varepsilon, \sigma^{k}) - \langle \varepsilon - D(u^{k+1}), \sigma^{k+1}
 \rangle \geq B(D(u^{k+1}), \sigma^{k})$$
 and we get the relation (\ref{maineq1}). 
\end{proof}
 
The previous theorem contains at part (II) an algorithm for finding a solution
of the problem (Pdisc). This algorithm is the rigorous reformulation of an
algorithm proposed in \cite{bergadesaxce} section 8. 

However, this theorem can be improved (and will be, in further research) in
several respects. Firstly, in the case of Dr\"{u}cker-Prager plasticity, the 
 function $\Delta b_{k}$ takes also infinite values. In this case the algorithm
 for solving the problem (Pdisc) should take the following form. Let $K$ denote the set of
 plastically admissible stresses. Then: 
 \begin{enumerate}
 \item[0.] initialize $\displaystyle (u^{0}, \sigma^{0})$ (for example take them
 equal to $(0,0)$, 
 \item[1.] repeat: given $\displaystyle (u^{k}, \sigma^{k}) \in U \times Y$, 
 \begin{enumerate}
  \item[a.](global condition) find $\displaystyle u^{k+1}$ such that 
  for all $v \in U$ 
  $$ B(D(u^{k+1}), \sigma^{k}) - \langle D(u^{k+1}), \sigma^{k} \rangle \leq 
  B(D(v), \sigma^{k}) - \langle D(v), \sigma^{k} \rangle$$ 
  \item[b.](local condition) define the stress $\displaystyle \sigma^{k+1}$ 
  almost everywhere in $\Omega$ by the relation $$\displaystyle \sigma^{k+1} 
  \in \pi_{K} \left( \partial   b(\cdot , \sigma^{k})(D(u^{k+1})) \right)$$  
  \end{enumerate}
  \end{enumerate}

 We don't know yet how to prove that such a sequence converges to a solution of
 the problem (\ref{maineq1}), which is the weak form of problem (Pdisc). 
 
 Secondly, by exploiting the particular expression of the functions $b$ which
 appear in real plasticity problems, we may be able to prove that sequences 
 $\displaystyle (u^{k}, \sigma^{k})$  have convergent subsequences, for example
 by a boundedness argument. 
 
 Another, potentially very interesting subject, concerns Coulomb friction. This
 law can be expressed by a bipotential, \cite{saxfeng} \cite{bipo3}. It should be interesting to
 explore the corresponding  variational formulation, where the bifunctional will
 contain volume integrals as well as surface integrals. Related to this see
 also the paper \cite{laborde}.

\end{document}